\documentclass[initials,citation-order,VANCOUVER]{WileyNJD-v1}
\usepackage{lmodern}
\usepackage{mathtools}
\usepackage{amssymb}
\usepackage{rndtable,uflstats,filltable,inertiatable}
\hypersetup{unicode=true,pdfusetitle,
 bookmarks=true,bookmarksnumbered=false,bookmarksopen=false,
  breaklinks=false,pdfborder={0 0 0},pdfborderstyle={},backref=false,colorlinks=false}

\makeatletter
\def\smashsmash{\relax
    \ifmmode\def\next{\mathpalette\mathsm@sh}\else\let\next\makesm@sh
    \fi \next}
\makeatother

\makeatletter

\newenvironment{lyxcode}
{\par\begin{list}{}{
\setlength{\rightmargin}{\leftmargin}
\setlength{\listparindent}{0pt}% needed for AMS classes
\raggedright
\setlength{\itemsep}{0pt}
\setlength{\parsep}{0pt}
\normalfont\ttfamily}%
 \item[]}
{\end{list}}

\DeclareMathOperator{\mathspan}{span}
\DeclareMathOperator{\diag}{diag}

\articletype{Research Article}%
\received{XXX}
\revised{XXX}
\accepted{XXX}
\begin{document}
\title{Wilkinson's Inertia-Revealing Factorization and Its Application to Sparse Matrices}
\author[1]{Alex Druinsky}
\author[1]{Eyal Carlebach}
\author[1]{Sivan Toledo}
\authormark{DRUINSKY \textsc{et al}}
\address[1]{\orgdiv{School of Computer Science}, \orgname{Tel Aviv University}, \orgaddress{Tel Aviv 69978, \country{Israel}}}
\corres{Sivan Toledo,
School of Computer Science,
Schreiber Building, Room 013,
Tel Aviv University,
Tel Aviv 69978,
Israel. \email{\href{mailto:stoledo@tau.ac.il}{stoledo@tau.ac.il}}}

\abstract[Summary]{We propose a new inertia-revealing factorization for sparse symmetric matrices. The factorization scheme and the method for extracting the inertia from it were proposed in the 1960s for dense, banded, or tridiagonal matrices, but they have been abandoned in favor of faster methods. We show that this scheme can be applied to any sparse symmetric matrix and that the fill in the factorization is bounded by the fill in the sparse $QR$ factorization of the same matrix (but is usually much smaller). We describe our serial proof-of-concept implementation, and present experimental results, studying the method's numerical stability and performance.}

\keywords{sparse matrix factorization, matrix inertia, symmetric indefinite matrices}

\jnlcitation{\cname{%
\author{Druinsky A},
\author{Carlebach E}, and
\author{Toledo S}}
(\cyear{2017}),
\ctitle{Wilkinson's Inertia-Revealing Factorization
and Its Application to Sparse Matrices}, \cjournal{Numer Linear Algebra Appl}, \cvol{XXX}.}

\maketitle

\section{Introduction}

The inertia of a real symmetric matrix $A$ is a triplet of numbers
that count how many of $A$'s eigenvalues are positive, negative or
zero, respectively. The first two of these numbers are called the
positive and negative indices of inertia, respectively, and the last
one is the nullity of the matrix. The inertia is easy to obtain by
computing all of the eigenvalues of $A$. However, even for moderately
large sparse problems, computing all of the eigenvalues is either
expensive or infeasible. Over the years, inertia algorithms have been
developed that do not resort to computing the full spectrum. Such
algorithms are the focus of this paper.

A fast inertia algorithm can facilitate a number of computational
tasks. As a first example, consider an eigenvalue of $A$ that is
isolated from the rest of the spectrum within a known interval $[x_{0},x_{1})$.
Let us see how to compute the multiplicity of that eigenvalue by taking
advantage of a fast inertia algorithm. The key is the property that
the negative index of inertia of the matrix $A-xI$, where $x$ is
real, is the number of eigenvalues of $A$ that lie to the left of
$x$. To compute the multiplicity, we compute the negative indices
of inertia of the matrices $A-x_{1}I$ and $A-x_{0}I$, and subtract
the latter index from the former. This yields the number of eigenvalues
that are to the left of $x_{1}$ but not to the left of $x_{0}$,
which is exactly the multiplicity of the target eigenvalue. 

The same idea works if we need to compute the aggregate multiplicity
of all of the eigenvalues within a known interval. This is useful
when we compute all of the eigenvalues within that interval using
a Lanczos shift-and-invert eigensolver, and we wish to ascertain that
no eigenvalues have been missed.\cite{zhang07}

Another example is when we need to compute a histogram of the spectrum
within an interval of interest. In this case, we divide the interval
into subintervals and compute the number of eigenvalues within each
subinterval using the above technique. Such a histogram is necessary
when we wish to compute all of the eigenvalues within the interval
of interest, and we need to partition the work evenly among multiple
processors.\cite{aktulga14} Here, each inertia computation is independent
from the others, which makes this scheme suitable for parallel computation.

The final example is bisection. In bisection, we compute the eigenvalues
of~$A$ by using the negative index of inertia of $A-xI$ as a binary-search
oracle for locating the eigenvalues by their ordinals (for an example,
see Golub and Van Loan\cite{golub13}*{Sec~8.4.1}).

To facilitate the above (and similar) tasks, we propose a new algorithm
that computes the inertia of sparse symmetric matrices. Our new sparse
inertia algorithm combines two old and obsolete algorithmic ideas
into a useful and novel algorithm. One idea is an algorithm that was
proposed by Wilkinson\cite{wilkinson65}*{p~236--240} for computing the inertia of dense, banded,
or tridiagonal matrices. Wilkinson
recommended the algorithm as the building block of a bisection eigensolver
for tridiagonal matrices. Over the years, this algorithm was replaced
by the use of $LDL^{T}$ factorizations of shifts of the matrix (and
by other tridiagonal eigensolvers; see~\cites{demmel97,parlett98,stewart01}).
The newer alternatives are more efficient and more stable; Wilkinson's
inertia idea appeared to be obsolete. The other idea is the row-by-row
sparse $QR$ factorization by George and Heath.\cite{george80}
It was replaced by sparse multifrontal $QR$ factorizations\cite{liu86}
(for a recent implementation, see~\cite{davis11a}). We have discovered
that George and Heath's approach to sparse factorizations and its
analysis can be applied to Wilkinson's inertia algorithm. This discovery
is what motivates the new algorithm.

Existing algorithms for computing the inertia of sparse matrices work
by computing a so-called symmetric-indefinite factorization (see~\cite{duff83}
for an early example). Such algorithms work well in practice, but
they are heuristic in the sense that no guarantees can be made about
the amount of work and memory that they require on any individual
matrix. In contrast, on many important classes of matrices, the new
algorithm provides an a priori bound on the fill that it generates
and hence on its storage requirements and (indirectly) on its running
time. A counterintuitive feature of the algorithm is that it computes
a nonsymmetric factorization; we sacrifice symmetry, but we gain bounds
on the fill.

Our implementation of the method is inspired by CSparse, Tim Davis's
concise sparse-matrix library.\cite{davis06} We felt that it was
premature to invest the kind of engineering effort that is required
to produce production-quality sparse-matrix factorization software,
and so we translated the method into code in a straightforward way,
opting for simplicity whenever possible. This means that the resulting
code is serial and does not provide optimal performance, but its performance
is predictable and easy to understand to anyone who understands the
basic algorithm.

The rest of the paper is organized as follows. We present the required
background material in Section~\ref{sec:background}. Section~\ref{sec:algorithm}
describes our sparse implementation and its effect on the sparsity
pattern of the matrix. We analyze the numerical stability of the algorithm
in Section~\ref{sec:numerical-analysis}, and present numerical experiments
in Section~\ref{sec:numerical-experiments}. The performance of our
implementation is studied experimentally in Section~\ref{sec:performance-experiments},
followed by our conclusions from this research in Section~\ref{sec:discussion-and-conclusions}.

\section{\label{sec:background}Background}

\subsection{\label{subsec:general-purpose-eigensolvers}General-Purpose Eigensolvers}

Eigenvalues of symmetric dense matrices are almost always computed
using a two-step procedure.\cite{golub13}*{Sec~8} In the first
step, we reduce the matrix to tridiagonal form using a sequence of
orthogonal similarity transformations. This reduction has the form
$Q^{T}AQ=T$, where $Q$ is orthogonal and $T$ is tridiagonal. The
reduction preserves the spectrum of $A$ and allows us to compute
$A$'s eigenvalues by applying a tridiagonal eigensolver to $T$;
this is the second step.

The computational cost of reducing an $n$-by-$n$ matrix $A$ to
$T$ is $\Theta(n^{3})$ arithmetic operations, and the cost of tridiagonal
eigensolvers such as bisection, the implicit $QR$ method or the more
recent MRRR algorithm\cite{dhillon06} is $\Theta(n^{2})$. Because
of the upfront investment in reducing the matrix to $T$, there is
no advantage in computing a subset of the eigenvalues; the whole spectrum
is produced at essentially the same cost. Eigenvectors can also be
computed: this costs $\Theta(n^{2})$ for the eigenvectors of $T$
and $\Theta(n^{3})$ to transform them into eigenvectors of $A$ by
reversing the effect of the orthogonal similarities.

\subsection{Lanczos}

The orthogonal-tridiagonalization approach is not suitable for sparse
matrices because orthogonal similarity transformations fill the matrix
with nonzeros, resulting in a computational cost of $\Theta(n^{3})$
even when $A$ is very sparse. The most widely applicable approach
for sparse symmetric eigenproblems is the Lanczos iteration. Here
we compute an orthogonal basis $Q$ of the Krylov subspace $K_{t}=\mathspan(b,Ab,\dotsc,A^{t-1}b)$,
where $t$ is the dimension of the subspace and $b$ is a nonzero
starting vector, usually chosen at random. We then use $Q$ to orthogonally
project $A$ along the Krylov subspace, producing a $t$-by-$t$ tridiagonal
matrix $T$ whose eigenvalues serve as approximations to the eigenvalues
of $A$. A popular implementation of this approach is the implicitly
restarted Lanczos method, available in the ARPACK library.\cites{calvetti94,lehoucq97}

Lanczos produces each column of $Q$ by multiplying $A$ with the
previous column, and then orthogonalizing the resulting vector against
the existing columns. The computational cost is $t-1$ matrix-vector
products and $\Theta(t^{2}n)$ arithmetic operations for the orthogonalization
operations. Typical sparse matrices can be of dimension $10^{6}$
or higher, and so letting~$t$ approach $n$ or even $\sqrt{n}$
would make the computational cost prohibitive. For this reason, the
number of iterative steps is usually chosen to be much smaller than~$n$.
Stopping after a small number of iterations means that only a few
of $A$'s eigenvalues are well approximated, typically those at the
edges of the spectrum.

When we wish to compute eigenvalues in the interior of the spectrum,
we compute a triangular factorization of $A-\sigma I$, where $\sigma$
is the point of interest in the spectrum, and iterate with $\left(A-\sigma I\right)^{-1}$
instead of iterating with $A$. Under this transformation, every eigenvalue
$\lambda$ of $A$ is mapped to an eigenvalue $(\lambda-\sigma)^{-1}$
in the transformed matrix's spectrum. The eigenvalues that are closest
to $\sigma$ in $A$'s spectrum map to the outermost eigenvalues of
$(A-\sigma I)^{-1}$, which leads Lanczos to converge to those eigenvalues
first. This is called shift-and-invert Lanczos.

The drawback of Lanczos is that it is hard to use as a black-box solver.
Convergence depends on the structure of the spectrum, and there is
no way to ascertain that all of the required eigenvalues have been
obtained. Another difficulty stems from the fact that the projection
of $K_{t}$ along any eigenspace of~$A$ is one dimensional, and
therefore Lanczos cannot determine the multiplicities of the eigenvalues
of $A$. If we wish to determine multiplicities, we can use a variant
called block Lanczos, but the costs rise with the dimension of the
computed eigenspace, which limits the efficiency of this approach.
Some of these deficiencies are addressed by algorithms popular in
the electronic-structure calculations community, such as conjugate
gradients and FEAST.\cites{ovtchinnikov08,polizzi09}

\subsection{\label{subsec:bisection}Bisection}

Bisection is a recursive algorithm that computes all of the eigenvalues
that lie in a user-specified half-open interval $\left[x_{0},\;x_{1}\right)$.
Since computing an interval that envelops the entire spectrum is relatively
easy (e.g., using the bound $|\lambda|\leq\|A\|_{2}\leq\|A\|_{1}$;
see~\cite{golub13}*{Thm~7.2.1}), the requirement for a user-provided
interval does not limit the generality of the method. Bisection works
by splitting the interval into two halves, counting the eigenvalues
in each half, and continuing the search recursively in each half that
contains eigenvalues. To count the eigenvalues in an interval~$\left[x_{0},\;x_{1}\right)$,
we determine the numbers $f(x_{0})$ and $f(x_{1})$ of eigenvalues
that lie strictly to the left of $x_{0}$ and of $x_{1}$; the difference
$f(x_{1})-f(x_{0})$ is the number of eigenvalues in the interval.
The function $f(x)$ is evaluated by computing the negative index
of inertia $\nu$ of $A-xI$, 
\[
f(x)=\nu(A-xI)\,.
\]

The recursion stops when the length $x_{1}-x_{0}$ of the current
interval falls below a user-specified tolerance threshold. At this
point, the algorithm returns the midpoint of the interval as a representative
of the eigenvalues within that interval. The midpoint is duplicated
as many times as there are eigenvalues. (Duplication can represent
either tight clusters of distinct eigenvalues or multiple eigenvalues).

A pseudocode of bisection is given in Algorithm~\ref{alg:bisection}.
The code can be easily modified to compute only the eigenvalues inside
a user-specified interval or a set of eigenvalues that are specified
by their ordinals.

\begin{algorithm}
\caption{\label{alg:bisection}Computing all of the eigenvalues using bisection.}
\begin{algorithmic}[1]
\algnewcommand{\LineComment}[1]{$\triangleright$ #1}
\Function{compute-all-eigenvalues}{$A,\tau$}
    \State \LineComment Computes all of the eigenvalues of a symmetric $n$-by-$n$ matrix $A$ to accuracy $\tau\|A\|_{1}$.
    \State \LineComment Returns a vector of $n$ computed eigenvalues.
    \State $x_{0}\;\gets\;-\left\Vert A\right\Vert _{1}$
    \State $\phantom{x_{0}}\mathllap{\nu_{0}}\;\gets\;0$
    \State $\phantom{x_{0}}\mathllap{x_{1}}\;\gets\;\left\Vert A\right\Vert _{1}$
    \State $\phantom{x_{0}}\mathllap{\nu_{1}}\;\gets\;n$
    \State \textbf{return} \Call{compute-within-interval}{$A,\tau,x_{0},\nu_{0},x_{1},\nu_{1}$}
\EndFunction
\State
\Function{compute-within-interval}{$A,\tau,x_{0},\nu_{0},x_{1},\nu_{1}$}
    \State \LineComment Recursively computes the eigenvalues that lie within $[x_{0},x_{1}).$ Returns a vector of the computed eigenvalues.
    \State \LineComment The arguments $\nu_{0}$ and $\nu_{1}$ are the negative indices of inertia of $A-x_{0}I$ and $A-x_{1}I$, respectively.
    \State $x\gets0.5(x_{0}+x_{1})$
    \If{$x_{1}-x_{0}>2\tau\left\Vert A\right\Vert _{1}$}
        \State $\mu\gets\nu(A-xI)$\Comment Compute the negative index of inertia of $A-xI$.\label{pseudocode:negative-index-of-inertia}
        \If{$\mu>\nu_{0}$}\label{pseudocode:left-half-interval}
            \State $w_{0}\gets$\Call{compute-within-interval}{$A,\tau,x_{0},\nu_{0},x,\mu$}
        \EndIf
        \If{$\nu_{1}>\mu$}\label{pseudocode:right-half-interval}
            \State $w_{1}\gets$\Call{compute-within-interval}{$A,\tau,x,\mu,x_{1},\nu_{1}$}
        \EndIf
        \State \textbf{return }$[w_{0},w_{1}]$\Comment Return the concatenation of $w_{0}$ and $w_{1}$.
    \Else
        \State \textbf{return }$[x,x,\dotsc,x]$\Comment Return the value $x$ repeated $\nu_{1}-\nu_{0}$ times.
    \EndIf
\EndFunction
\end{algorithmic}
\end{algorithm}

The computational cost of bisection is as follows. Every node of the
recursion tree requires computing $\nu(A-xI)$, whose cost we denote
by $N(A)$. The number of nodes in the tree is bounded by the number
of leaves times the height of the tree; the number of leaves equals
at most the number of returned eigenvalues $k$ and the height of
the tree equals $\Theta(\log(1/\tau))$, where $\tau$ is the error
tolerance threshold. Therefore the total cost is $O(k\log(1/\tau)N(A))$.

As we explained in Section~\ref{subsec:general-purpose-eigensolvers},
we compute the eigenvalues of a dense matrix by first reducing it
orthonormally to a tridiagonal one. This allows us to compute $\nu(A-xI)$
using an $LDL^{T}$ factorization without pivoting at an arithmetic
cost of $O(n)$ operations. Using the $LDL^{T}$ factorization to
compute the inertia of a tridiagonal matrix is numerically stable
even without pivoting.\cite{demmel93} In LAPACK, this is implemented
in the subroutine STEBZ.\cites{anderson99,kahan66} This approach
gives a total arithmetic cost of $\Theta(n^{3}+k\log(1/\tau)n)$,
where the $n^{3}$ term is the cost of the tridiagonal reduction. 

When the matrix is sparse, we must either compute an $LDL^{T}$ factorization
with pivoting using a sparse symmetric-indefinite factorization algorithm
(several such algorithms exist; we address them next), or use our
new algorithm. In either case, bisection is not competitive with shift-and-invert
Lanczos for the canonical problem of accurately computing the eigenvalues
that are closest to a user-selected point of interest. For this problem,
Lanczos requires a single sparse factorization, while bisection requires
a sequence of $\Omega(\log(1/\tau))$, even when computing a single
eigenvalue. Nevertheless, bisection has two advantages that make it
potentially useful when the required accuracy is moderate. First,
bisection decouples the work that takes place within disjoint intervals
of the spectrum, making it suitable for parallel computation, and
second, bisection can locate eigenvalues by their ordinals, which
to our knowledge is not possible with other eigensolvers.

\subsection{\label{subsec:sparse-symmetric-indefinite}Sparse Symmetric-Indefinite
Factorizations}

An alternative to our sparse inertia algorithm is to use a sparse
symmetric-indefinite factorization. Such factorizations have the form
$PAP^{T}=LDL^{T}$, where $P$ is a permutation matrix and $D$ is
a block-diagonal one with 1-by-1 and 2-by-2 diagonal blocks. The inertia
of $D$ is trivial to compute because each of its diagonal blocks
corresponds to one or two easy-to-compute eigenvalues. By Sylvester's
Law of Inertia (see~\cite{golub13}*{Thm~8.1.17}), the matrices
$D$ and~$A$ have the same inertia, and therefore having $D$ allows
us to easily compute the inertia of $A$. Examples of modern symmetric-indefinite
factorization codes include MA57, HSL\_MA86, HSL\_MA97,\cite{hsl}
and the sparse symmetric-indefinite solver from the PARDISO library.\cite{schenk06}

\subsection{\label{subsec:sparse-factorizations}Sparse Factorizations and Fill
Bounds}

In the Cholesky factorization of a symmetric positive-definite matrix
$A=LL^{T}$, the fill in $L$ has a tight bound that is fully determined
by the sparsity pattern of~$A$. The fill is best described in terms
of a graph that we define as follows. The vertex set is $\{1,2,\ldots,n\}$,
representing the rows and columns of the matrix, and the edge set
contains the $(i,j)$ pairs that correspond to the nonzero elements
of~$A$. A seminal (and fairly easy) result shows that if $L_{ij}\neq0$,
then there is a path in the graph from $i$ to $j$ consisting entirely
of vertices with indices smaller than $i$ and $j$.\cite{rose76}
This is a necessary condition. It is normally also sufficient; only
exact cancellation can lead to $L_{ij}=0$ in the presence of such
a path. 

The key to minimizing the fill in Cholesky is to order the vertices
so as to break up fill-creating paths. Instead of factorizing $A$,
we compute a factorization $PAP^{T}=LL^{T}$, in which $P$ is a permutation
matrix that represents the vertex ordering. This is the idea behind
a method called nested dissection\cite{george73} (see also~\cite{davis06}*{Sec~7.6}),
which we use in our experiments below. Nested dissection generates
$P$ by finding a small set of vertices called an \emph{approximately
balanced vertex separator}, ordering that set last, and recursing
on two subgraphs. This method provides powerful guarantees on the
fill. For example, if the graph of $A$ is a $\sqrt{n}$-by-$\sqrt{n}$
finite-element mesh or a planar graph of $n$ vertices, ordering by
nested dissection guarantees that the Cholesky factor contains $O(n\log n)$
nonzero elements and the factorization requires $O(n^{3/2})$ arithmetic
operations, which is asymptotically optimal. More general optimality
bounds\cites{agrawal93,gilbert88,lipton79} for nested dissection have also been proved.\footnote{The optimality of nested dissection requires the use of nearly optimal
vertex separators. In practice, vertex separators are computed using
heuristics. These heuristics usually work, but they are not guaranteed
to find a small separator, even if one exists. Nested dissection can
fail to produce an approximately optimal ordering if the separator
heuristic fails.}

The fill in the sparse $QR$ factorization can be analyzed and controlled
using the same techniques as in Cholesky. The $R$ factor of the $QR$
factorization of $A$ is also the Cholesky factor of~$A^{T}A$, and
therefore the fill of the $R$ factor can be controlled by computing
a nested-dissection ordering for $A^{T}A$ and reordering the columns
of~$A$ accordingly. George and Ng have analyzed this and proved
that the sparse $QR$ factorization of a possibly unsymmetric or indefinite
$\sqrt{n}$-by-$\sqrt{n}$ finite-element mesh has the same asymptotic
costs as the Cholesky factorization of a symmetric positive-definite
matrix of this type.\cite{george88} Furthermore, a nested-dissection
ordering of $A^{T}A$ can be computed efficiently without computing
$A^{T}A$ or even its sparsity structure. Recent algorithms that compute
such orderings\cites{brainman02,grigori10,hu05} include hypergraph-partitioning-based nested dissection,
wide-separator nested dissection and reduction to singly bordered
block-diagonal form.\footnote{Although the reduction to singly bordered block-diagonal form is generally
not considered to be a nested-dissection algorithm, we can think of
it as a multiway dissection algorithm for ordering $A^{T}A$ that
has only one level of nesting.}

The relevance of the sparse $QR$ factorization to this paper is that
the fill in our algorithm is bounded by the fill of sparse $QR$.
Due to this, our factorization falls into the class of algorithms
in which an upper bound on the fill is established by using a fill-minimizing
algorithm before the computation begins. 

In contrast, sparse symmetric-indefinite factorization algorithms
work by ordering the matrix using a fill-minimizing ordering for the
Cholesky factorization of $A$, and then computing an $LDL^{T}$ factorization
with pivoting. Pivoting is necessary for numerical stability, but
can potentially increase fill. The PARDISO library addresses this
tradeoff using a technique called static pivoting, in which the matrix
is preprocessed to reduce the need for pivoting, and pivots are chosen
so as not to affect fill. In certain cases, badly chosen pivots can
make the factorization numerically unstable. The codes in the HSL
library support static pivoting too, but they also offer the possibility
of selecting pivots without regard to fill. This can prevent numerical
issues but can increase fill. 

\subsection{Wilkinson's Inertia Algorithm}

We base our algorithm on a method that Wilkinson proposed for computing
the inertia of general matrices,\cite{wilkinson65}*{p~236--240}
which has also been adapted for banded matrices.\cites{gupta69,gupta70,martin67,scott84}
The method is based on the Sturm sequence property, which is described
by the following theorem. The theorem states that the number of negative
eigenvalues of a matrix can be computed by forming the sequence of
the determinants of its leading principal minors and counting the
number of sign changes between consecutive elements of that sequence.
The proof of the theorem is based on the eigenvalue interlacing property
(for details, see~\cite{wilkinson65}*{p~300}).
\begin{theorem}[Sturm Sequence Property]
The number of negative eigenvalues of a symmetric matrix $A\in\mathbb{R}^{n\times n}$
equals the number of sign changes between consecutive elements of
the sequence 
\[
1,\det(A_{1}),\det(A_{2}),\dotsc,\det(A_{n})\,,
\]
 where $A_{k}=A_{1:k,1:k}$ are the leading principal submatrices
of $A$ for $k=1,2,\dotsc,n$, assuming that none of the elements
of that sequence are zero. 
\end{theorem}
We can, therefore, compute the number of negative eigenvalues by forming
the sequence of the determinants of the leading principal minors.
Wilkinson proposed to form this sequence by reducing the matrix to
upper-triangular form one row after another, and producing the determinant
of the corresponding leading principal minor as a byproduct after
processing each row. Let us describe this idea in concrete form using
matrix notation. The first $k$ steps of the algorithm reduce the
first $k$ rows of the matrix to a $k$-by-$n$ upper-trapezoidal
matrix $U^{(k)}$. The transformation that we apply to $A_{1:k,:}$
in these steps can be described in aggregate using a $k$-by-$k$
matrix $X^{(k)}$, which is not formed by the algorithm and is computed
only implicitly. Using this notation, we may write 
\[
X^{(k)}A_{1:k,:}=U^{(k)}\,.
\]
 Let us take the determinant of the leading $k$-by-$k$ principal
submatrix on both sides of this equation. This yields
\[
\det(X^{(k)}A_{1:k,1:k})=\det(U_{1:k,1:k}^{(k)})\,,
\]
 which we may write as
\[
\det(X^{(k)})\det(A_{k})=U_{11}^{(k)}U_{22}^{(k)}\dotsm U_{kk}^{(k)}\,,
\]
where $U_{11}^{(k)},U_{22}^{(k)},\dotsc,U_{kk}^{(k)}$ are the diagonal
elements of $U^{(k)}$. This gives 
\[
\det(A_{k})=\frac{U_{11}^{(k)}U_{22}^{(k)}\dotsm U_{kk}^{(k)}}{\det(X^{(k)})}\,.
\]
Computing the numerator in the above formula is trivial. To explain
how we compute $\det(X^{(k)})$, we need to explain how $U^{(k)}$
is produced from $U^{(k-1)}$.

Let the reduced matrix $A^{(k)}$ that we obtain from step $k$ be
defined as the matrix that contains the rows of the upper-trapezoidal
$U^{(k)}$, followed by the yet-unprocessed rows of $A$. For $n=6$
and $k=4$, this matrix has the form
\[
A^{(k)}=\!\!\!\begin{array}{r}
\vphantom{\begin{array}{c}
\times\\
\times\\
\times\\
\times
\end{array}}U^{(k)}\\
\vphantom{\begin{array}{c}
\times\\
\times
\end{array}}A_{(k+1):n,:}
\end{array}\!\!\!\left[\begin{array}{cccccc}
\times & \times & \times & \times & \times & \times\\
 & \times & \times & \times & \times & \times\\
 &  & \times & \times & \times & \times\\
 &  &  & \times & \times & \times\\
\cmidrule(lr){1-6}\times & \times & \times & \times & \times & \times\\
\times & \times & \times & \times & \times & \times
\end{array}\right].
\]
 To compute $U^{(k)}$, we eliminate the elements in row $k$ of $A^{(k-1)}$
one by one from left to right using a sequence of elimination operations,
stopping when we reach the diagonal element. We can choose between
two types of elimination operations: elementary stabilized transformations
and Givens rotations. An elementary stabilized transformation works
by eliminating $A_{kj}^{(k-1)}$ in two steps:
\begin{enumerate}
\item Test whether $\bigl|A_{jj}^{(k-1)}\bigr|<\bigl|A_{kj}^{(k-1)}\bigr|$,
and if so, interchange rows $k$ and $j$.
\item Subtract a multiple of row $j$ from row $k$ using $A_{kj}^{(k-1)}/A_{jj}^{(k-1)}$
as the scaling factor.
\end{enumerate}
A Givens rotation works by premultiplying rows $k$ and $j$ with
the rotation matrix 
\[
G^{(k,j)}=\frac{1}{\sqrt{\bigl(A_{jj}^{(k-1)}\bigr)^{2}+\bigl(A_{kj}^{(k-1)}\bigr)^{2}}}\left[\begin{array}{cc}
A_{jj}^{(k-1)} & A_{kj}^{(k-1)}\\
-A_{kj}^{(k-1)} & A_{jj}^{(k-1)}
\end{array}\right].
\]

Having completed the $k-1$ elimination operations (if $A$ is sparse,
fewer than $k-1$ may be necessary), we can now compute $\det(A_{k}).$
If we use Givens rotations, $\det(X^{(k)})=1$. The determinant of
a rotation matrix always equals $1$, the determinant of a product
of rotations is therefore also 1, and so $\det(X^{(k)})=1$. If we
use elementary stabilized transformations, $\det(X^{(k)})=\pm1$,
and resolving the sign is straightforward. In every elimination step,
we perform up to two operations: an interchange of a pair of rows,
which corresponds to a permutation matrix with determinant $-1$,
and a subtraction of a multiple of one row from a subsequent row,
which corresponds to a unit lower-triangular matrix whose determinant
is 1. Therefore, when using elementary stabilized transformations,
$\det(X^{(k)})$ is either $1$ or $-1$, depending on whether the
number of row interchanges that we have made is even or odd.

For completeness, we list a pseudocode of the variant of the algorithm
that uses elementary stabilized transformations in Algorithm~\ref{alg:pseudocode-dense}.
In the pseudocode, we simplified the algorithm by counting the number
of sign changes in the determinant sequence without actually computing
the determinants.

\begin{algorithm}
\caption{\label{alg:pseudocode-dense}Computing the number of negative eigenvalues using elementary stabilized transformations.}
\begin{algorithmic}[1]
\State \textbf{Input: }a symmetric matrix $A\in\mathbb{R}^{n\times n}$.
\State \textbf{Output: }the number $\nu$ of negative eigenvalues of $A$.
\State
\State $\nu\gets0$
\For{$i\gets1,2,\dotsc,n$}
    \State $x\gets0$ \Comment number of row interchanges and diagonal sign changes
    \For{$j\gets1,2,\dotsc,i-1$}
        \If{$\left|A_{jj}\right|<\left|A_{ij}\right|$}
            \State $A_{j,:}\leftrightarrow A_{i,:}$
            \State $x\gets x+1$
            \If{$\operatorname{sign}(A_{jj})\neq\operatorname{sign}(A_{ij})$}
                \State $x\gets x+1$
            \EndIf
        \EndIf
        \State $A_{i,:}\gets A_{i,:}-\frac{A_{ij}}{A_{jj}}A_{j,:}$
    \EndFor
    \If{$A_{ii}<0$}
        \State $x\gets x+1$
    \EndIf
    \If{$x\bmod2=1$}
        \State $\nu\gets\nu+1$
    \EndIf
\EndFor
\end{algorithmic}
\end{algorithm}

\section{\label{sec:algorithm}The New Sparse Inertia Algorithm}

\subsection{Bounds on Fill and Arithmetic Cost}

In the Givens rotations-based variant of the factorization, the computation
is identical to that of George and Heath's sparse $QR$ factorization,
which factors the matrix row by row.\cite{george80} The fill in
this case is obviously identical to that of sparse $QR$, and the
number of arithmetic operations is the same as in George and Heath's
algorithm. 

If we use elementary stabilized transformations in Wilkinson's algorithm,
the fill that the algorithm produces is bounded by the fill in George
and Heath's $QR$ factorization. 
\begin{definition}
Element $A_{ij}^{(k)}$ in position $i,j$ of the $k$th reduced matrix
in some factorization algorithm is called a \emph{structural nonzero
}if it is different from zero whenever the algorithm is carried out
in an arithmetic that satisfies the following for all scalars $x$
and $y$:
\end{definition}
\begin{enumerate}
\item If $x+y=0$ or $x-y=0$, then $x=y=0$.
\item $0x=0$.
\end{enumerate}
\begin{theorem}
If $\smash{A_{ij}^{(k)}}$ is a structural nonzero in the variant
of Wilkinson's algorithm that uses elementary stabilized transformations,
then $\smash{A_{ij}^{(k)}}$ is also a structural nonzero in the variant
that uses Givens rotations and in George and Heath's $QR$ factorization.
\end{theorem}
\begin{proof}
We prove the result by induction on $k$. Clearly, the statement is
true for~$k=1$, because no transformations have taken place. In
step~$k$, when an elementary stabilized transformation eliminates
$A_{kj}^{(k-1)}$, it either leaves the structure of row $j$ unchanged,
or it replaces that structure with the one of row $k$ if an interchange
of rows is necessary. The structure of row $k$ is then replaced with
the union of the two row structures, with the exception of element
$(k,j)$, which becomes zero. In a Givens rotation, the structure
of both rows is replaced with their union, again with the exception
of $(k,j)$. Therefore the sparsity structure that we obtain from
a Givens rotation serves as an upper bound on the one that we obtain
from an elementary stabilized transformation. Figure~\ref{fig:fill-bound}
illustrates the two types of elimination steps. 
\end{proof}
\begin{figure}
\begin{centering}
\subfloat[Elementary stabilized transformation]{\includegraphics{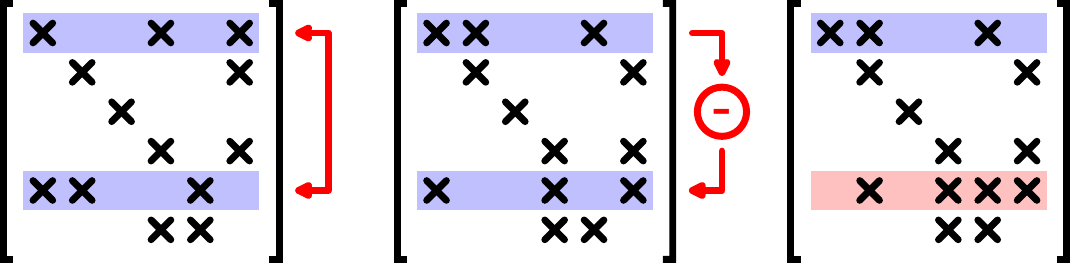}

}
\par\end{centering}
\begin{centering}
\vspace{1ex}
\subfloat[Givens rotation]{\includegraphics{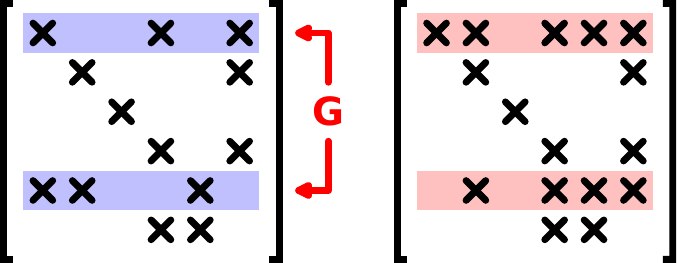}

}
\par\end{centering}
\caption{\label{fig:fill-bound}A comparison of the effect that an elementary
stabilized transformation has on the sparsity structure to the effect
of a Givens rotation. Here, $k=5$, $j=1$, and the elementary stabilized
transformation exchanges rows $k$ and $j$ before eliminating element
$(k,j)$.}
\end{figure}

The stabilized-elementary variant performs fewer arithmetic operations
than the $QR$ variant. In the dense setting, elementary stabilized
transformations require three times fewer arithmetic operations (because
each elementary transformation is cheaper to perform than the corresponding
Givens rotation). When the matrix is sparse, the effect can be much
more dramatic, because fill that is not created in the lower triangle
does not need to be eliminated (there are fewer transformations to
apply) and because sparser rows mean that each transformation is cheaper.

We acknowledge that operation counts are not good predictors of running
times on modern computers. To achieve fast running times, it is essential
to exploit parallelism and to reduce communication. In sparse factorization,
these optimizations are usually performed by exploiting elimination-tree
parallelism,\cite{jess82} and by invoking dense subroutines on
submatrices that are full or almost full (supernodes\cite{liu93}).
Our implementation of the algorithm does not currently include such
optimizations; implementing them is a significant research/engineering
effort which we have not undertaken (see, e.g.,~\cites{chen08,duff04,gupta00,hogg10,hogg11,li05}).
Nevertheless, our implementation does demonstrate the sparsity of
the algorithm and suggests that a high-performance variant of this
algorithm can achieve performance comparable to that of other modern
sparse codes.

\subsection{A Sparse Implementation}

Our implementation of the algorithm follows the outline of Algorithm~\ref{alg:pseudocode-dense},
modified so that we store and operate only on the nonzero elements
of the matrix. We represent the matrix using a data structure that
we call \emph{expandable compressed sparse row} (ECSR), which is a
modification of the standard compressed sparse row (CSR) layout.\cite{davis06}*{p~8}
ECSR exploits the fact that the fill in our algorithm is bounded by
the fill in sparse $QR$, and that the number of structural nonzeros
in each row in the $R$ factor can be computed in almost linear time.\cite{gilbert94}
Our code uses the function \texttt{rowcolcounts} of the CHOLMOD library
to compute the row-wise nonzero counts.\cite{chen08}

ECSR stores the nonzero elements and their column indices one row
after another. The amount of space that we allocate to row $i$ is
equal to the computed number of nonzeros in $R_{i,:}$, the $i$th
row in the matrix's $R$ factor (if $A_{i,:}$ is denser than $R_{i,:}$,
we allocate enough space for $A_{i,:}$). This guarantees that the
space allocated to row $i$ is sufficient to store row $i$ of $R$
and of all the intermediate upper-trapezoidal matrices. The representation
consists of the following four vectors:
\begin{description}
\item [{\texttt{val}}] numerical values of the nonzeros (in all the rows)
\item [{\texttt{col}}] column indices of the nonzeros (in all the rows)
\item [{\texttt{head}}] pointers into \texttt{val }and \texttt{col} indicating
where each row of the matrix starts
\item [{\texttt{tail}}] pointers indicating where each row ends.
\end{description}
The vacant space of row $i$ is stored in positions \texttt{tail{[}i{]}}
through \texttt{head{[}i + 1{]}}. 

We use one of two additional data structures to perform row eliminations.
In the Givens variant of the algorithm, we use a two-row Sparse Accumulator
(SPA),\cite{gilbert92} and in the elementary-stabilized variant,
we use an Ordered Sparse Accumulator (OSPA).\cite{gilbert99} A
multiple-row SPA represents a set of $m$ sparse rows with a shared
sparsity structure using an $m$-by-$n$ matrix \texttt{spa.val},
which holds the numerical values of the nonzeros, a length-$n$ vector
\texttt{spa.occupied} whose elements indicate which columns of \texttt{spa.val}
are nonzero, and also \texttt{spa.col}, a list of the column indices
of the nonzeros. The OSPA is similar, except that \texttt{spa.col}
is stored in a heap data structure,\cite{cormen09}*{Sec~6} which
provides cheap access to the leftmost nonzero column.

In the Givens variant of the algorithm, we use the SPA to perform
Givens rotations as follows.
\begin{enumerate}
\item We copy the two rows into the SPA one by one. We start with the nonzero
elements of the first row, updating the corresponding elements of
\texttt{spa.val} and \texttt{spa.occupied} and adding the column indices
of the nonzeros into \texttt{spa.col}. Next we copy the second row
in the same way, except that now we use \texttt{spa.occupied} to avoid
duplicating column indices in \texttt{spa.col}.
\item We apply the rotation to each nonzero column, using the column indices
in \texttt{spa.col} to enumerate the nonzero columns.
\item Finally, we copy the nonzero elements back into the ECSR data structure
and restore the SPA to its initial state, using \texttt{spa.col} to
enumerate elements of \texttt{spa.occupied} and \texttt{spa.val} that
must be cleared.
\end{enumerate}
In this implementation, the total amount of work is proportional to
the number of floating-point operations that the Givens rotation performs.

The elementary-stabilized variant is more complex. We implemented
it using several subroutines that operate on the ECSR data structure
and on a one-row~OSPA:
\begin{description}
\item [{\texttt{load}}] Loads a row of an ECSR matrix into an empty OSPA.
This is similar to loading a row into an unordered SPA, but requires
a \texttt{build-heap} operation\cite{cormen09}*{Section~6.3} on
the array of column indices.
\item [{\texttt{store}}] Copies the contents of an OSPA back into a row
of an ECSR matrix.
\item [{\texttt{retrieve-head}}] Returns the position and the numerical
value of the leftmost nonzero element of an OSPA (but leaves it in
the OSPA). If the OSPA is empty, returns a special indicator, which
we denote $\bot$.
\item [{\texttt{remove-head}}] Removes the leftmost nonzero element from
an OSPA.
\item [{\texttt{subtract}}] Subtracts a multiple of a row of an ECSR matrix
from an OSPA. This is a routine operation on a SPA; on an OSPA, columns
need to be inserted into the column array using a \texttt{heap-insert}
operation.
\item [{\texttt{swap}}] Exchanges the content of an OSPA with a row of
an ECSR matrix.
\end{description}
Given these building blocks, the overall algorithm can be composed
as shown in Algorithm~\ref{alg:pseudocode-sparse}.

\begin{algorithm}
\caption{\label{alg:pseudocode-sparse}The sparse stabilized-elementary transformations variant of the algorithm.}
\begin{algorithmic}[1]
\State \textbf{Input: }a symmetric ECSR matrix $A\in\mathbb{R}^{n\times n}$.
\State \textbf{Output: }the number $\nu$ of negative eigenvalues of $A$.
\State
\State $\nu\gets0$
\State $s\gets\text{empty OSPA of length }n$
\For{$i\gets1,2,\dotsc,n$}
    \State $x\gets0$
    \State \texttt{load(}$s$, $A_{i,:}$\texttt{)}
    \State $\left(j,A_{ij}\right)\gets\texttt{retrieve-head(}s\texttt{)}$
    \While{$j\neq\bot$ and $j<i$}
        \If{$\left|A_{jj}\right|<\left|A_{ij}\right|$}
            \State \texttt{swap(}$s$, $A_{j,:}$\texttt{)}
            \State $x\gets x+1$
            \If{$\operatorname{sign}(A_{jj})\neq\operatorname{sign}(A_{ij})$}
                \State $x\gets x+1$
            \EndIf
        \EndIf
        \State \texttt{subtract(}$s$, $\left(A_{ij}/A_{jj}\right)\cdot A_{j,:}$\texttt{)}
        \State \texttt{remove-head(}$s$\texttt{)}
        \State $\left(j,A_{ij}\right)\gets\texttt{retrieve-head(}s\texttt{)}$
    \EndWhile
    \If{$A_{ii}<0$}
        \State $x\gets x+1$
    \EndIf
    \If{$x\bmod2=1$}
        \State $\nu\gets\nu+1$
    \EndIf
    \State \texttt{store(}$s$, $A_{i,:}$\texttt{)}
\EndFor
\end{algorithmic}
\end{algorithm}

\section{\label{sec:numerical-analysis}Numerical Analysis}

Next, we describe the numerical behavior of the algorithm. The core
numerical issue is explained in Wilkinson's book.\cite{wilkinson65}*{p~312--315}
We repeat this explanation below, for completeness, and also provide
an example that shows how this numerical issue manifests itself.

The Givens variant of the algorithm is a thoroughly studied $QR$
factorization scheme and is known to compute backward-stable factors
(see~\cite{higham02}*{Sec~19}). The elementary-stabilized version
is not as well known, but it has also been analyzed in the literature.
It has been called \emph{pairwise pivoting} and was analyzed by Sorensen.\cite{sorensen85}
Sorensen's analysis proves that the normwise backward error of the
corresponding factorization is bounded by a product of four numbers:
the magnitude of the elements of $A$, the unit roundoff $u$ ($2^{-53}\approx1.1\times10^{-16}$
in double-precision IEEE 754), and two growth factors. One of these
growth factors represents the norm of the computed upper-triangular
factor, and the other represents the norm of the implicit factor that
is the equivalent of the $L$ factor in the $LU$ factorization with
partial pivoting. In contrast with $LU$ with partial pivoting, this
implicit factor is not lower triangular, and its elements are not
bounded by~1 in magnitude. Sorensen proves in his paper that the
two growth factors are bounded by $2^{n-1}-1$ and $2^{n-1}$, respectively.
He also notes that although these bounds are large, they are highly
pessimistic, and that large growth factors have not been encountered
in practice. Further experimental evidence of this was provided by
Grigori, Demmel and Xiang.\cite{grigori11}*{Fig~4.2}

\subsection{An Obstacle to a Backward-Stability Analysis}

The backward stability of the $QR$ factorization and of pairwise
pivoting implies that the factors that we obtain from our algorithm
are backward stable. This applies to the final factors, but also to
the intermediate ones that we obtain in each step. Stability can be
hampered if the factorization suffers from growth, but large growth
is rare. Because the algorithm uses the factors to compute the signs
of the determinant sequence
\[
1,\det(A_{1}),\det(A_{2}),\dotsc,\det(A_{n})\,,
\]
 the factors' backward stability implies that each sign is backward
stable. However, this does not imply the backward stability of the
computed inertia. This is because each intermediate factor is modified
in the steps that follow the step in which it is obtained, and therefore
each sign of the determinant sequence has its own individual backward-error
matrix (that equals the difference between $A$ and a matrix $\tilde{A}$
for which the factor is exact). 

There is no reason to assume that if each factor is the exact factor
of a matrix close to $A$, then they are all the factors of \emph{one}
matrix close to $A$, which would imply backward stability of the
determinant sequence. Indeed, we can show that the algorithm is sometimes
not backward stable.

\subsection{An Example of Instability}

Next, we show an example of how the algorithm can fail. We provide
the transcript of a Matlab session, to make the example concrete.

We let $A$ be a square, symmetric $n$-by-$n$ matrix of the form
\[
A = \,
        \begin{array}{c}
            \raisebox{1pt}{$\scriptstyle n/2$} \\
            \raisebox{1pt}{$\scriptstyle n/2$}
        \end{array}
        \!\left[
        \begin{array}{cc}
            \smashsmash{\overset{\text{\Large\strut}\raisebox{1pt}{$\scriptstyle n/2$}}{\vphantom{Z^T} X}} &
            \smashsmash{\overset{\text{\Large\strut}\raisebox{1pt}{$\scriptstyle n/2$}}{Z^T}} \\
            Z & 0 \\
        \end{array}
        \right]
\]
The leading diagonal block has the form 
\[
X=Q\,\diag(1,\epsilon_{1},\epsilon_{2},\dotsc,\epsilon_{n/2-1})\;Q^{T}\,,
\]
where $Q$ is a random orthogonal matrix, and $\epsilon_{1},\epsilon_{2},\dotsc,\epsilon_{n/2-1}$
are random, independent, and normally distributed scalars with mean
0 and a standard deviation that equals the unit roundoff. The notation
$\diag(\,\cdot\,)$ indicates a diagonal matrix with the specified
diagonal elements. The block $Z$ is defined so that its elements
are random, independent, and normally distributed with mean 0 and
standard deviation 1. 

In Matlab, we construct this matrix as follows:
\begin{lyxcode}
>\textcompwordmark{}>~n~=~2048;

>\textcompwordmark{}>~Q~=~orth(randn(n~/~2));

>\textcompwordmark{}>~d~=~{[}1;~eps~{*}~randn(n~/~2~-~1,~1){]};

>\textcompwordmark{}>~X~=~Q~{*}~diag(d)~{*}~Q';

>\textcompwordmark{}>~Z~=~randn(n~/~2);

>\textcompwordmark{}>~A~=~{[}X~Z';~Z~zeros(n~/~2){]};~
\end{lyxcode}
Next, we compute $A$'s eigenvalues by running the Matlab command
\texttt{eig}. Using these eigenvalues, we compute the matrix's condition
number and its negative index of inertia:
\begin{lyxcode}
>\textcompwordmark{}>~ev~=~eig(A);

>\textcompwordmark{}>~kappa~=~max(abs(ev))~/~min(abs(ev))

kappa~=~3.4498e+03

~

>\textcompwordmark{}>~nu~=~sum(ev~<~0)

nu~=~1024
\end{lyxcode}
In this example, the \texttt{eig} command calls LAPACK's symmetric
eigensolver,\cite{anderson99} which is backward stable. Its backward
stability guarantees that the computed eigenvalues, condition number,
and negative index of inertia are those of some matrix~$\tilde{A}$
close to $A$. As we see above, the condition number is small, and
so all of the matrices in the vicinity of $\tilde{A}$, including
$A$ itself, have eigenvalues with the same signs as~$\tilde{A}$.
This means that all such matrices have the same negative index of
inertia, namely~$1{,}024$.

The inertia algorithm produces a different result:
\begin{lyxcode}
>\textcompwordmark{}>~our\_nu~=~inertia(A)

our\_nu~=~1026
\end{lyxcode}
This shows that the algorithm is not backward stable. Any backward-stable
algorithm would return the negative index of inertia of a matrix in
the vicinity of $A$ and $\tilde{A}$, which must be $1{,}024$.

The result is similar regardless of whether we use the Givens variant
of the algorithm or the elementary-stabilized one. In particular,
because growth is not possible in the Givens variant, this shows that
the instability is not caused by growth.

\subsection{Instability Is Caused by Near Singularity}

The cause of the instability is singularity of the leading principal
minors. In our~$A$, almost all of the leading principal submatrices
have a condition number of order~$u^{-1}$, where $u$ is the unit
roundoff. The only exceptions are the three submatrices of orders
1, $n-1$ and~$n$. Because the other $n-3$ submatrices are nearly
singular, the computed signs of their determinants are often incorrect,
even when computed using a backward-stable algorithm, and therefore
counting the number of sign changes in the determinant sequence yields
an incorrect result.

\subsection{The (Limited) Effect of the Instability on Bisection}

In the experiments that we describe in the next section, the eigenvalues
that we produce using bisection are usually as accurate as those that
we produce using a backward-stable eigensolver. The instability of
the inertia algorithm does not appear to have an effect on bisection.

The reason for this is that the instability manifests itself only
when $A-xI$ has a leading principal submatrix that is nearly singular.
To make one of these submatrices nearly singular, $x$ must be at
a distance of $O(u)\Vert A\Vert$ from an eigenvalue of such a submatrix.
This can happen under two scenarios. First, as bisection converges
to an eigenvalue of $A$, the point $x$ moves progressively closer
to that eigenvalue until it reaches within $O(u)\Vert A\Vert$. This
scenario does not create a problem, because when $x$ is close to
an eigenvalue of $A$, an error in the computed inertia cannot prevent
bisection from converging.\cite{wilkinson65}*{p~302--305} The
other scenario is that $x$ is close to an eigenvalue of a leading
principal submatrix of $A$, and that eigenvalue is far from every
eigenvalue of $A$. This can lead bisection to fail. However, such
unfortunate positioning of $x$ appears to happen only in pathological
cases, meaning that the instability has few opportunities to have
an effect.

\section{\label{sec:numerical-experiments}Numerical Experiments}

We carried out extensive numerical experiments to test the accuracy
of the algorithm. In these experiments, we used our inertia algorithm
as part of a bisection eigensolver to compute the eigenvalues of a
variety of matrices. Although in practice bisection is often not a
competitive algorithm, it requires a large number of inertia computations,
and this makes it a convenient platform to stress-test our algorithm.

We computed the eigenvalues of two families of matrices. The first
family is one that we produced using the LAPACK subroutine LATMS,
which generates random dense matrices for testing LAPACK subroutines.
These matrices have the form $A=Q\Lambda Q^{T}$, where $Q$ is a
random orthogonal matrix and $\Lambda$ is a diagonal matrix whose
diagonal elements $\lambda_{1},\lambda_{2},\dotsc,\lambda_{n}$ are
the eigenvalues of $A$. The eigenvalues are randomly distributed
according to two parameters named \texttt{mode} and $\kappa$. There
are six possible modes, in the first five of which the eigenvalues
have the form $\lambda_{i}=s_{i}\sigma_{i}$, where the scalars $s_{i}$
take the values $\pm1$ independently with equal probability, and
the scalars $\sigma_{i}$ have a distribution that depends on the
mode. In all of the distributions that correspond to the first five
modes, the $\sigma_{i}$ are nonnegative, and therefore they are the
singular values of $A$. The distributions are the following:
\begin{enumerate}
\item The first singular value is equal to~1, and all others are equal
to $1/\kappa$.
\item The first $n-1$ singular values are equal to~1, and $\sigma_{n}=1/\kappa$.
\item The singular values are evenly distributed between~1 and $1/\kappa$,
on a logarithmic scale.
\item The singular values are evenly distributed between~1 and $1/\kappa$,
on a linear scale.
\item The singular values have the form $\left(1/\kappa\right)^{e}$, where
$e$ is drawn independently for each singular value from the uniform
distribution on the interval $\left(0,1\right)$.
\end{enumerate}
Finally, in mode 6, LATMS generates the $\lambda_{i}$ directly without
generating the singular values first. In this mode, the $\lambda_{i}$
are independent and normally distributed with mean~0 and standard
deviation~1.

We generated one matrix of order $n=256$ for each combination of
\texttt{mode} and~$\kappa$, for various values of $\kappa$, and
computed all of the eigenvalues of that matrix. We then measured their
accuracy using the maximal normalized error, 
\begin{equation}
\max_{i=1,2,\dotsc,n}|\lambda_{i}-\hat{\lambda}_{i}|/\Vert A\Vert_{1}\,,\label{eq:eigenvalue-error}
\end{equation}
where the $\lambda_{i}$ are the actual eigenvalues used by LATMS
and the $\hat{\lambda}_{i}$ are the approximate ones produced by
our bisection eigensolver. The full results are shown in Table~\ref{tab:experiments-latms}.
All of the errors are within one order of magnitude from $u\approx1.1\times10^{-16}$.

\begin{table}%
\caption{\label{tab:experiments-latms}Accuracy of the computed eigenvalues of random dense matrices, generated using the LAPACK subroutine~LATMS.}%
\centering \rndtable%
\end{table}%

The second family of matrices in our experiments are sparse matrices
from the SuiteSparse Matrix Collection.\cite{davis11b} We first
computed the inertia of all~283 real symmetric matrices of order
$64<n<16{,}384$ in this collection whose elements have known numerical
values. Each inertia computation was carried out on a single core
of an Intel i7-2600 CPU, with our code compiled using version~13.1.1
of the Intel compiler suite. We then used this computation as follows
to winnow down our list of matrices. The analysis of bisection in
Section~\ref{subsec:bisection} indicates that the time required
to compute all of the eigenvalues of $A$ to double-precision accuracy
is at most $-\log(u)\,n\,T(A)$, where~$u$ is double-precision unit
roundoff, $n$ is the order of $A$, and $T(A)$ is the time required
to compute $A$'s inertia once. We selected all of the matrices for
which this estimate was at most 2 hours, and ran bisection to compute
all of the eigenvalues of each such matrix. In total, we ran bisection
on 116 matrices, described below in Table~\ref{tab:ufl-matrices}.

As in the previous experiment, we computed the eigenvalues of each
of the~116 matrices, and measured the maximal normalized error~(\ref{eq:eigenvalue-error})
relative to backward-stable eigenvalues computed using the LAPACK
subroutine SYEV. The results are shown in a histogram plot in Figure~\ref{fig:experiments-ufl}.
The errors are again comparable to $u$, ranging between $8.0\times10^{-16}$
and $3.5\times10^{-14}$, with a median of $3.5\times10^{-15}$.

\begin{table}%
\caption{\label{tab:ufl-matrices}Statistical parameters of the set of SuiteSparse matrices.}%
\centering \uflstats%
\end{table}

\begin{figure}
\begin{centering}
\includegraphics{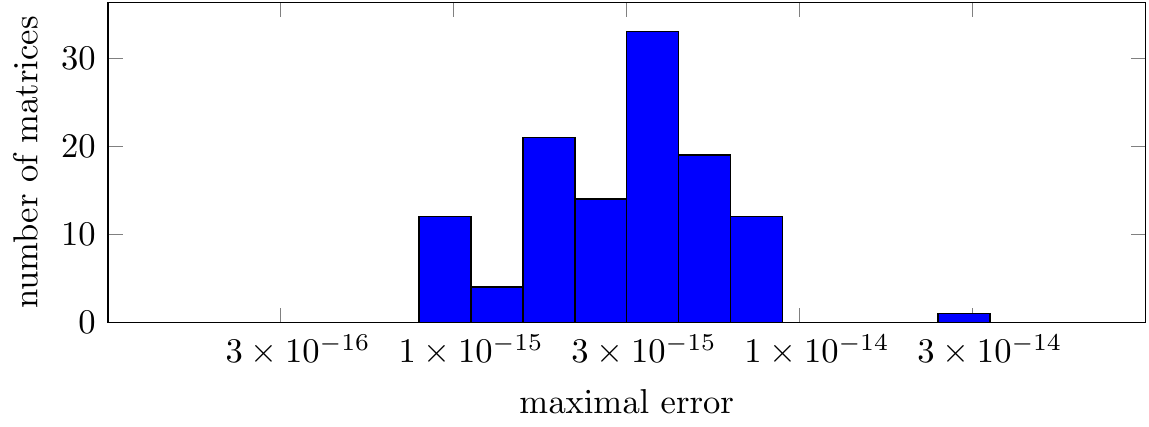}
\par\end{centering}
\caption{\label{fig:experiments-ufl}Accuracy of the computed eigenvalues of
sparse matrices from the SuiteSparse collection.}
\end{figure}

\section{\label{sec:performance-experiments}Performance Experiments}

In this section, we describe the experiments that we carried out to
study the performance of our sparse inertia algorithm. We analyzed
the performance of the two variants of our algorithm, and we compared
it with the performance of two other factorization algorithms. The
first algorithm that we compared ourselves with is SuiteSparseQR,
or SPQR for short.\cite{davis11a} It is a state-of-the-art sparse
$QR$ factorization algorithm, and it is the algorithm that runs when
we apply Matlab's \texttt{qr} command to a sparse matrix. Our interest
in SPQR stems from the fact that it computes the same factorization
that we compute in the Givens variant of our algorithm. Comparing
ourselves with SPQR helps us to learn how much performance we can
gain by adopting the advanced techniques that are implemented in SPQR,
such as the use of supernodes.

The second algorithm with which we compared ourselves is the algorithm
MA57 from the HSL library.\cites{duff04,hsl} It is an industrial-quality
sparse-fac\-tor\-i\-za\-tion algorithm for symmetric-indefinite
matrices. An important distinction between SPQR and MA57 is that the
former does not compute the inertia, while the latter does. In these
experiments, MA57 represents the alternative to our approach to the
problem of computing the inertia. A comparison with MA57 allows us
to gain insight into the relative strengths and weaknesses of our
algorithm.

We carried out two experiments on a single core of an Intel Xeon E5-2650
v3 CPU, using gcc and gfortran version 6.2.0, Intel MKL version 11.3.1,
SuiteSparse version 4.5.3, and MA57 version 3.8.0. We use the default
settings for all algorithms. In particular, this means that MA57 pre-scales
the matrix so as to reduce the need for pivoting, which is something
our algorithm does not do.

Our first experiment was carried out on a set of 81 matrices from
the SuiteSparse Matrix Collection. The matrices were chosen by selecting
all of the symmetric, real, indefinite matrices of order $32{,}768\leq n\leq65{,}536$.
For each matrix, we computed a nested-dissection ordering for $A$
using METIS, and then applied each of the algorithms to the reordered
matrix. We discarded~21 matrices for which at least one factorization
took more than one hour or required more than 16 GB of memory; this
left us with a set of 60 matrices. 

Figure~\ref{fig:performance-givens-vs-spqr} compares the performance
of the Givens variant of our algorithm to that of SPQR. We find that
the required number of arithmetic operations (flops) in the two algorithms
is within a factor of 10 of each other, with no significant advantage
in favor of either one. The reason for the variation in flops is because
the two algorithms process the rows in different orders, and the order
of the rows can have a dramatic effect on flops. This has already
been noted by George and Heath.\cite{george80}*{p~78} The computational
rate of SPQR is however much higher, ranging between 1.1 and 10 gflop/s,
whereas that of our algorithm ranges between 151 and 874 mflop/s.
This also translates to a dramatic difference in running times in
favor of SPQR.

\begin{figure}
\noindent \centering{}\includegraphics{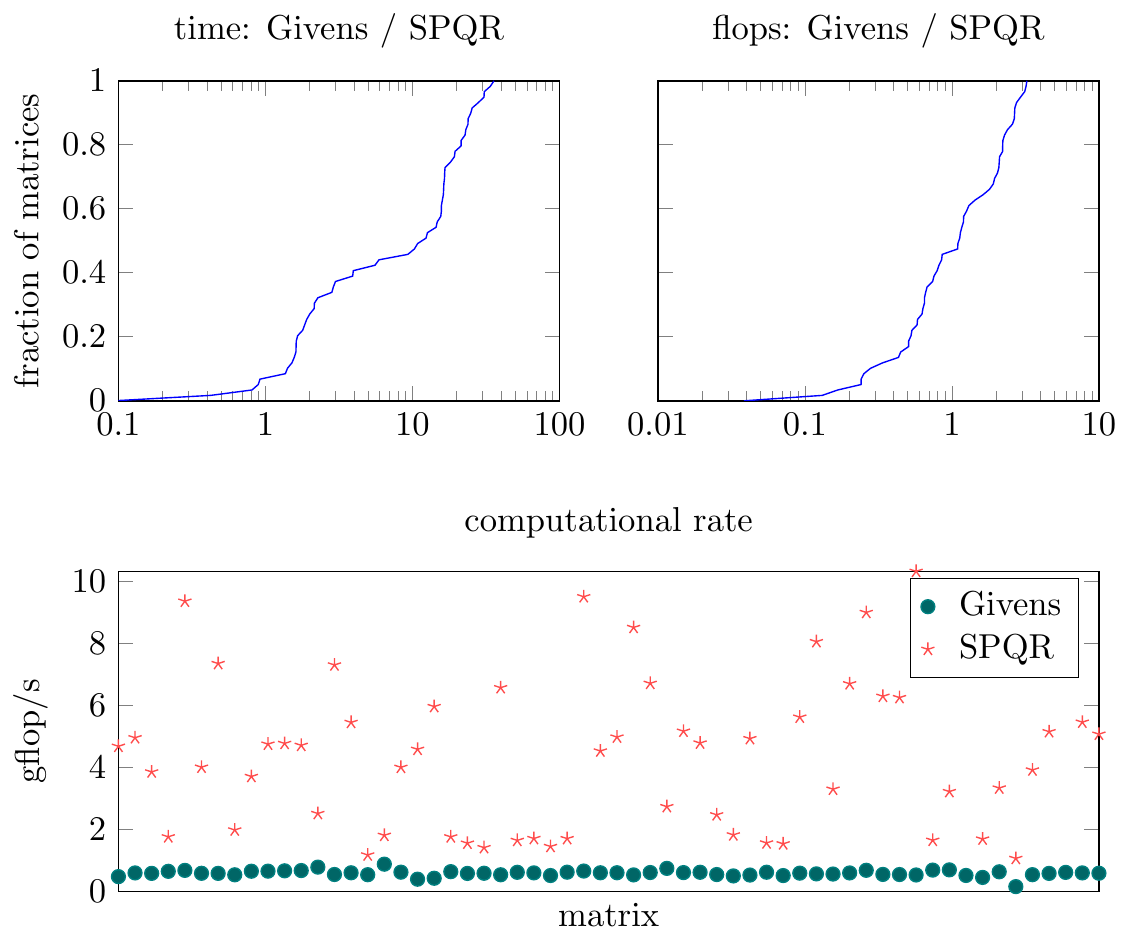}\caption{\label{fig:performance-givens-vs-spqr}A comparison of the Givens
variant of our algorithm with SPQR. The top plots show two empirical
cumulative-distribution-function curves of the ratio of the running
time of our algorithm to that of SPQR (left) and the ratio of the
numbers of arithmetic operations of the two algorithms (right). The
bottom plot shows the computational rates of the two algorithms.}
\end{figure}

Figure~\ref{fig:performance-elementary-vs-givens} shows that the
elementary-stabilized variant of the algorithm is much more efficient
than the Givens variant. The elementary-stabilized variant performs
a factor of 0.001 to 0.13 of the flops that the Givens variant performs.
The computational rates of the two codes are comparable, with some
outliers that we ascribe to the cost of row interchanges in the elementary-stabilized
case. The elementary-stabilized version is therefore faster by a significant
factor, between 3.14 and 269.

\begin{figure}
\noindent \centering{}\includegraphics{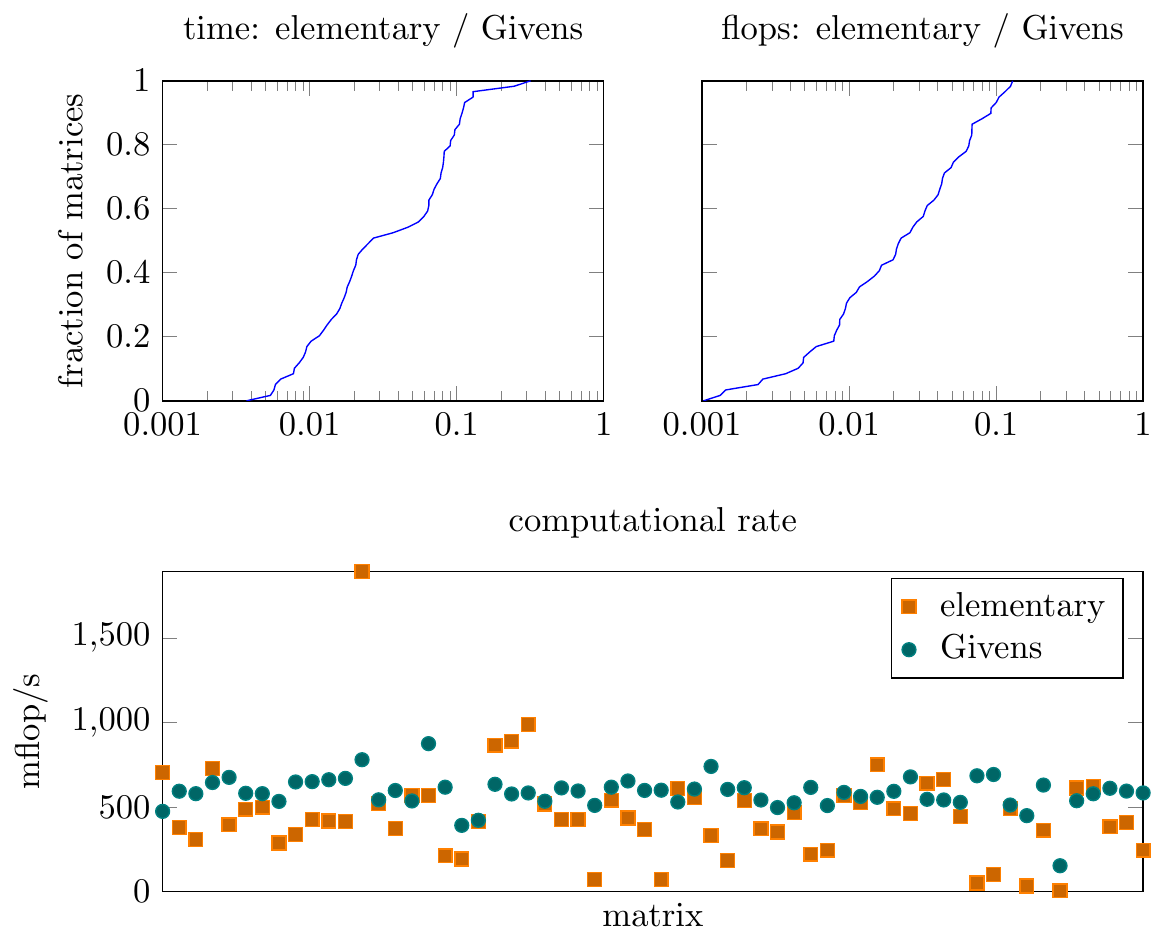}\caption{\label{fig:performance-elementary-vs-givens}A comparison of the two
variants of our algorithm using the same type of plots as in Figure~\ref{fig:performance-givens-vs-spqr}.}
\end{figure}

Figure~\ref{fig:performance-elementary-vs-spqr} compares the elementary-stabilized
variant with SPQR. Although, as we have seen in the previous figures,
the computational rate of our algorithm is orders of magnitude lower,
we preserve sparsity much better. Our algorithm always performs fewer
flops than SPQR, with a median ratio of 0.017. Due to the much better
sparsity, we are actually faster in 42 of the 60 matrices, with a
time ratio that ranges between 0.0011 and 3.64.

\begin{figure}
\noindent \centering{}\includegraphics{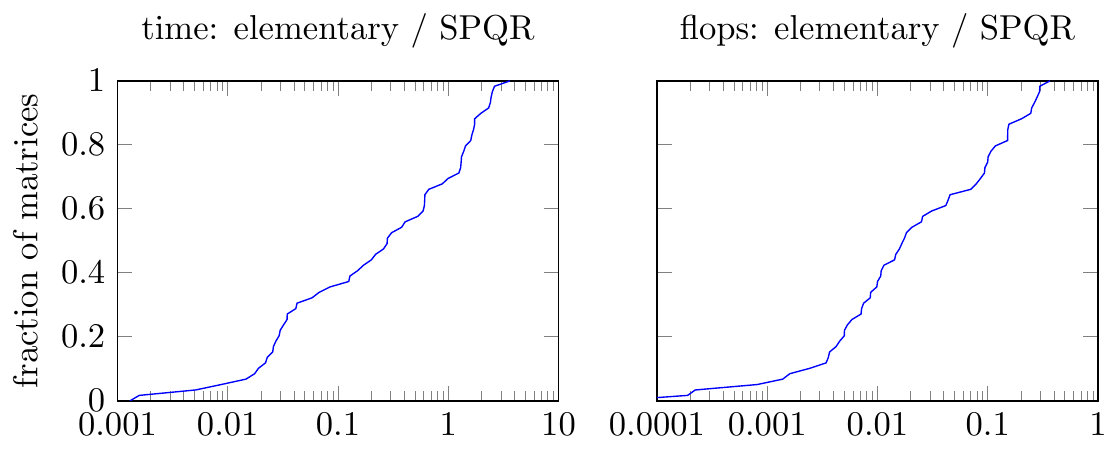}\caption{\label{fig:performance-elementary-vs-spqr}A comparison of the elementary-stabilized
variant of our algorithm with SPQR. (The computational rates of the
two algorithms are presented in Figures~\ref{fig:performance-givens-vs-spqr}
and~\ref{fig:performance-elementary-vs-givens}.)}
\end{figure}

Finally, Figure~\ref{fig:performance-elementary-vs-ma57} compares
the performance of the elementary-stabilized variant with MA57. We
see that MA57 is better than our algorithm at preserving sparsity,
requiring up to a factor of 21.7 fewer flops than our algorithm. MA57
also achieves a significantly faster computational rate of up to 18.55
gflop/s, although in 33 of the matrices it falls below 1 gflop/s and
can be as low as 51 mflop/s. Nevertheless, MA57 is faster in all but
3 matrices, with a speedup that ranges between 0.75 and 333.

\begin{figure}
\noindent \centering{}\includegraphics{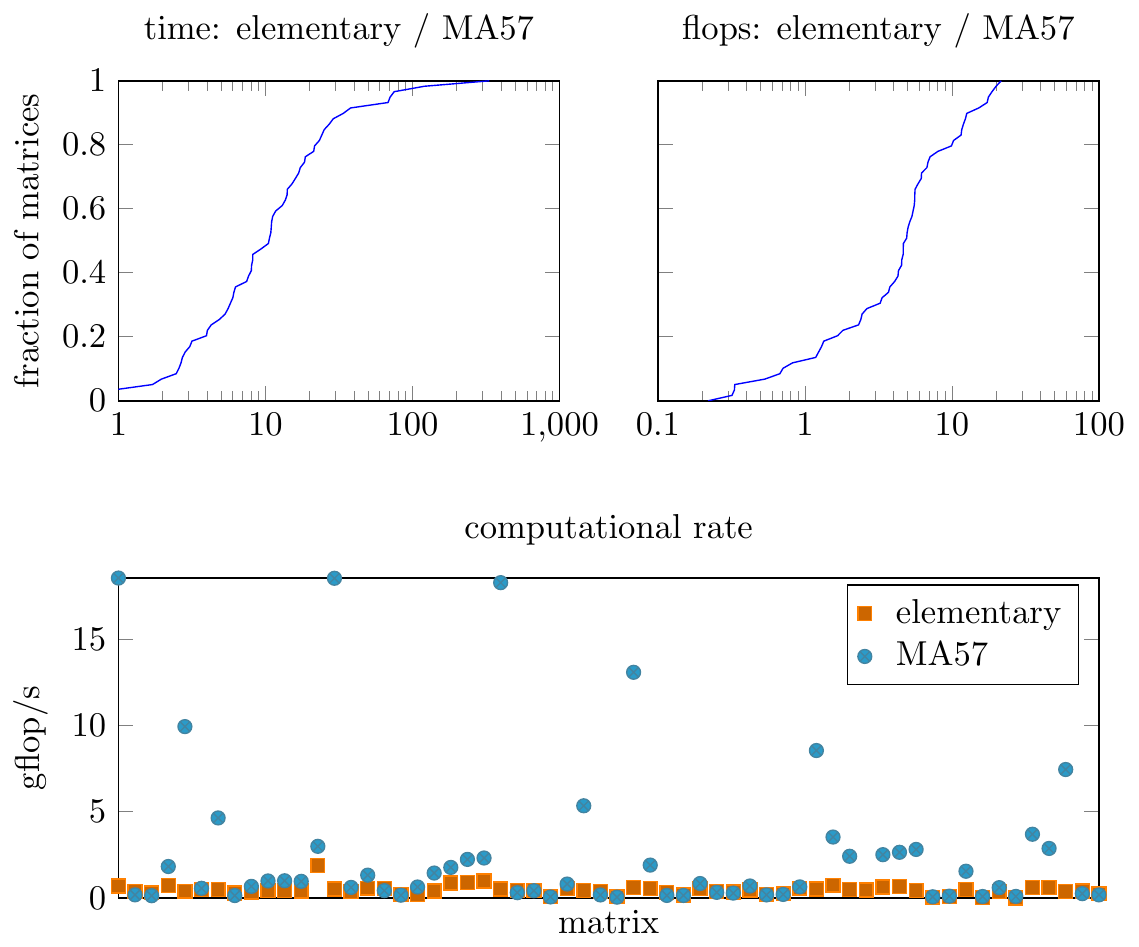}\caption{\label{fig:performance-elementary-vs-ma57} A comparison of the elementary-stabilized
variant of our algorithm with MA57.}
\end{figure}

In this experiment, we found two matrices for which our algorithm
and MA57 did not return the same inertia. The first matrix is AG-Monien/shock-9,
for which both versions of our algorithm returned a negative index
of inertia of 18,168, and MA57 returned 18,178. According to the annotation
in the SuiteSparse collection, this matrix has a nullspace of dimension
137, so the discrepancy is not surprising. The second matrix is AG-Monien/se,
which has a nullspace of dimension 457, and for which our elementary-transformation
based algorithm returned 16,215, while the Givens-rotation version
and MA57 returned 16,175.

In our final experiment, we made a more detailed comparison of our
algorithm with MA57. Our goal is to find matrices where our more conservative
approach pays off. We started with the 422 symmetric, real, indefinite
matrices of order $8{,}192\leq n\leq1{,}048{,}576$ in the SuiteSparse
collection, but kept only the 54 matrices whose MA57 factor had at
least twice as many nonzeros as the Cholesky factor would have contained
(had the matrix been positive definite). This produced a set of matrices
on which pivoting in MA57 caused significant fill. For each of the
54 matrices, we invoked our factorization and MA57 using two orderings:
an optimistic nested-dissection ordering of $A$ and a conservative
nested-dissection ordering of the structure of $A^{T}A$. We then
compared MA57 and our algorithm with respect to a parameter we refer
to as fill, which we define as the number of nonzeros in the triangular
factor, divided by the number of nonzeros in $A$.\footnote{In defining the fill of our factorization, we disregard the size of
its triangularizing factor. The reason for this is that this factor
is represented only implicitly in the algorithm and is not stored
in memory.} We omit the results for MA57 using an $A^{T}A$ ordering because
MA57 always produced less fill with an $A$ ordering than with an
$A^{T}A$ ordering. In 19 of the matrices we found that the fill in
our algorithm is lower than the fill of MA57, sometimes dramatically
so. A list of these matrices is given in Table~\ref{tab:fill-ma57}.

\begin{table}%
\caption{\label{tab:fill-ma57}Matrices for which our algorithm with one of the orderings produced less fill than MA57. We define fill as the number of nonzero elements in the triangular factor, divided by the number of nonzeros in $A$. Column headings nd and wide indicate nested-dissection ordering of $A$ and of $A^{T}A$ respectively (the latter of the two is equivalent to a wide-separator nested-dissection ordering of $A$). The symbol $\dagger$ indicates that the computation exceeded the maximal allocated time of one hour, and $\ddagger$ indicates that the algorithm required more memory than the allocated 16~GB.}%
\centering \filltable
\end{table}

As in the previous experiment, we found six matrices for which the
different algorithms did not return the same inertia. Of these matrices,
two are identified by the SuiteSparse matrix collection as having
a high-dimensional nullspace; three are large matrices for which the
dimension of the nullspace is not known, but all three are binary;
and finally, the matrix GHS\_index/bratu3 is nonsingular but has the
same features as the example of Section~\ref{sec:numerical-analysis}.
The details are in Table~\ref{tab:inertia-ma57}.
\begin{table}%
\caption{\label{tab:inertia-ma57}The computed negative index of inertia for matrices on which the two algorithms did not return the same value.}%
\centering \inertiatable%
\end{table}

\section{\label{sec:discussion-and-conclusions}Discussion and Conclusions}

In this paper, we proposed a sparse variant of Wilkinson's inertia
algorithm. The new inertia algorithm provides provable a priori bounds
on the fill, which other algorithms do not provide.

Our performance experiments show that even without using the techniques
that are necessary to obtain top performance from a sparse-matrix
algorithm (such as using supernodes), our algorithm is competitive
and often superior to a top-quality sparse $QR$ factorization code
due to the sparser factors that we produce in our algorithm.

Nevertheless, our algorithm is usually slower than MA57, often much
slower. Due to the use of supervariables in MA57, it achieves much
higher computational rates, and thanks to its use of threshold pivoting
and a matching algorithm that allows it to rescale the matrix so as
to prevent pivoting,\cite{duff05} it is better able to preserve
sparsity.\footnote{The recently released factorization code HSL\_MA97 is able not only
to rescale the matrix based on the matching, but also to reorder it.
This can potentially further reduce the need for pivoting. However,
computing the matching requires time, and the reordering can potentially
increase fill, and so the use of matching is recommended by the authors
of HSL only for matrices that require a great deal of pivoting.} For matrices that do not require significant pivoting, disabling
matching-based rescaling can further improve MA57's performance. Other
ordering methods, such as minimum degree, can produce less fill and
thus yield further gains. These techniques can be used in our algorithm
as well, and we expect significant improvement if they are used. Even
without this, we found several cases where our approach yields less
fill, and one matrix (Andrianov/net150) that MA57 cannot factor within
the allotted time and space, whereas we can.

\section*{Acknowledgments}

We thank the anonymous referee for pointing out several errors in
our description of MA57 and PARDISO, and for many other comments that
greatly helped to improve the paper.

This research was supported in part by grant 1045/09 from the Israel
Science Foundation (founded by the Israel Academy of Sciences and
Humanities) and by grant 2010231 from the US-Israel Binational Science
Foundation.

\bibliography{paper}

\begin{bibdiv}
\begin{biblist}

\bib{zhang07}{article}{
      author={Zhang, Hong},
      author={Smith, Barry},
      author={Sternberg, Michael},
      author={Zapol, Peter},
       title={{SIPs}: Shift-and-invert parallel spectral transformations},
        date={2007},
     journal={ACM T Math Software},
      volume={33},
      number={2},
}

\bib{aktulga14}{article}{
      author={Aktulga, Hasan~Metin},
      author={Lin, Lin},
      author={Haine, Christopher},
      author={Ng, Esmond~G.},
      author={Yang, Chao},
       title={Parallel eigenvalue calculation based on multiple shift-invert
  {Lanczos} and contour integral based spectral projection method},
        date={2014},
     journal={Parallel Comput},
      volume={40},
      number={7},
       pages={195\ndash 212},
}

\bib{golub13}{book}{
      author={Golub, Gene~H.},
      author={Van~Loan, Charles~F.},
       title={Matrix computations},
     edition={4},
   publisher={Johns Hopkins University Press},
        date={2013},
}

\bib{wilkinson65}{book}{
      author={Wilkinson, J.~H.},
       title={Algebraic eigenvalue problem},
   publisher={Clarendon Press, Oxford},
        date={1965},
}

\bib{demmel97}{book}{
      author={Demmel, James~W.},
       title={Applied numerical linear algebra},
   publisher={SIAM},
        date={1997},
}

\bib{parlett98}{book}{
      author={Parlett, Beresford~N.},
       title={The symmetric eigenvalue problem},
   publisher={SIAM},
        date={1998},
}

\bib{stewart01}{book}{
      author={Stewart, G.~W.},
       title={Matrix algorithms: Volume {II}: Eigensystems},
   publisher={SIAM},
        date={2001},
}

\bib{george80}{article}{
      author={George, Alan},
      author={Heath, Michael~T.},
       title={Solution of sparse linear least squares problems using {Givens}
  rotations},
        date={1980},
     journal={Linear Algebra Appl},
      volume={34},
       pages={69\ndash 83},
}

\bib{liu86}{article}{
      author={Liu, Joseph W.~H.},
       title={On general row merging schemes for sparse {Givens}
  transformations},
        date={1986},
     journal={SIAM J Sci Comput},
      volume={7},
      number={4},
       pages={1190\ndash 1211},
}

\bib{davis11a}{article}{
      author={Davis, Timothy~A.},
       title={Algorithm 915, {SuiteSparseQR}: Multifrontal multithreaded
  rank-revealing sparse {QR} factorization},
        date={2011},
     journal={ACM T Math Software},
      volume={38},
      number={1},
       pages={8:1\ndash 8:22},
}

\bib{duff83}{article}{
      author={Duff, I.~S.},
      author={Reid, J.~K.},
       title={The multifrontal solution of indefinite sparse symmetric linear
  equations},
        date={1983},
     journal={ACM T Math Software},
      volume={9},
      number={3},
       pages={302\ndash 325},
}

\bib{davis06}{book}{
      author={Davis, Timothy~A.},
       title={Direct methods for sparse linear systems},
   publisher={SIAM},
        date={2006},
}

\bib{dhillon06}{article}{
      author={Dhillon, Inderjit~S.},
      author={Parlett, Beresford~N.},
      author={V\"{o}mel, Christof},
       title={The design and implementation of the {MRRR} algorithm},
        date={2006},
     journal={ACM T Math Software},
      volume={32},
       pages={533\ndash 560},
}

\bib{calvetti94}{article}{
      author={Calvetti, D.},
      author={Reichel, L.},
      author={Sorensen, D.~C.},
       title={An implicitly restarted {Lanczos} method for large symmetric
  eigenvalue problems},
        date={1994},
     journal={Electron T Numer Ana},
      volume={2},
       pages={1\ndash 21},
}

\bib{lehoucq97}{book}{
      author={Lehoucq, R.~B.},
      author={Sorensen, D.~C.},
      author={Yang, C.},
       title={{ARPACK} users' guide},
   publisher={SIAM},
        date={1997},
}

\bib{ovtchinnikov08}{article}{
      author={Ovtchinnikov, E.~E.},
       title={Computing several eigenpairs of {Hermitian} problems by conjugate
  gradient iterations},
        date={2008},
     journal={J Comput Phys},
      volume={227},
      number={22},
       pages={9477\ndash 9497},
}

\bib{polizzi09}{article}{
      author={Polizzi, Eric},
       title={Density-matrix-based algorithm for solving eigenvalue problems},
        date={2009},
     journal={Phys Rev B},
      volume={79},
      number={11},
       pages={115112},
}

\bib{demmel93}{article}{
      author={Demmel, James~W.},
      author={Gragg, William},
       title={On computing accurate singular values and eigenvalues of matrices
  with acyclic graphs},
        date={1993},
     journal={Linear Algebra Appl},
      volume={185},
       pages={203\ndash 217},
}

\bib{anderson99}{book}{
      author={Anderson, E.},
      author={Bai, Z.},
      author={Bischof, C.},
      author={Blackford, S.},
      author={Demmel, J.},
      author={Dongarra, J.},
      author={Du~Croz, J.},
      author={Greenbaum, A.},
      author={Hammarling, S.},
      author={McKenney, A.},
      author={Sorensen, D.},
       title={{LAPACK} users' guide},
     edition={3},
   publisher={SIAM},
        date={1999},
}

\bib{kahan66}{techreport}{
      author={Kahan, W.},
       title={Accurate eigenvalues of a symmetric tri-diagonal matrix},
 institution={Computer Science Department, Stanford University},
        date={1966},
      number={CS41},
}

\bib{hsl}{misc}{
       title={{HSL}, a collection of {Fortran} codes for large-scale scientific
  computation},
        note={\url{http://www.hsl.rl.ac.uk/}},
}

\bib{schenk06}{article}{
      author={Schenk, Olaf},
      author={G\"{a}rtner, Klaus},
       title={On fast factorization pivoting methods for sparse symmetric
  indefinite systems},
        date={2006},
     journal={Electron T Numer Ana},
      volume={23},
       pages={158\ndash 179},
}

\bib{rose76}{article}{
      author={Rose, Donald~J.},
      author={Tarjan, R.~Endre},
      author={Lueker, George~S.},
       title={Algorithmic aspects of vertex elimination on graphs},
        date={1976},
     journal={SIAM J Comput},
      volume={5},
      number={2},
       pages={266\ndash 283},
}

\bib{george73}{article}{
      author={George, Alan},
       title={Nested dissection of a regular finite element mesh},
        date={1973},
     journal={SIAM J Numer Anal},
      volume={10},
      number={2},
       pages={345\ndash 363},
}

\bib{agrawal93}{incollection}{
      author={Agrawal, Ajit},
      author={Klein, Philip},
      author={Ravi, R.},
       title={Cutting down on fill using nested dissection: Provably good
  elimination orderings},
        date={1993},
   booktitle={Graph theory and sparse matrix computation},
      series={The {IMA} volumes in mathematics and its applications},
      volume={56},
   publisher={Springer},
       pages={31\ndash 55},
}

\bib{gilbert88}{article}{
      author={Gilbert, John~Russell},
       title={Some nested dissection order is nearly optimal},
        date={1988},
     journal={Inform Process Lett},
      volume={26},
      number={6},
       pages={325\ndash 328},
}

\bib{lipton79}{article}{
      author={Lipton, Richard~J.},
      author={Rose, Donald~J.},
      author={Tarjan, Robert~Endre},
       title={Generalized nested dissection},
        date={1979},
     journal={SIAM J Numer Anal},
      volume={16},
      number={2},
       pages={346\ndash 358},
}

\bib{george88}{article}{
      author={George, Alan},
      author={Ng, Esmond},
       title={On the complexity of sparse {$QR$} and {$LU$} factorization of
  finite-element matrices},
        date={1988},
     journal={SIAM J Sci Stat Comp},
      volume={9},
      number={5},
       pages={849\ndash 861},
}

\bib{brainman02}{article}{
      author={Brainman, Igor},
      author={Toledo, Sivan},
       title={Nested-dissection orderings for sparse {LU} with partial
  pivoting},
        date={2002},
     journal={SIAM J Matrix Anal A},
      volume={23},
      number={4},
       pages={998\ndash 1012},
}

\bib{grigori10}{article}{
      author={Grigori, Laura},
      author={Boman, Erik~G.},
      author={Donfack, Simplice},
      author={Davis, Timothy~A.},
       title={Hypergraph-based unsymmetric nested dissection ordering for
  sparse {LU} factorization},
        date={2010},
     journal={SIAM J Sci Comput},
      volume={32},
      number={6},
       pages={3426\ndash 3446},
}

\bib{hu05}{article}{
      author={Hu, Yifan},
      author={Scott, Jennifer},
       title={Ordering techniques for singly bordered block diagonal forms for
  unsymmetric parallel sparse direct solvers},
        date={2005},
     journal={Numer Linear Algebr},
      volume={12},
      number={9},
       pages={877\ndash 894},
}

\bib{gupta69}{article}{
      author={Gupta, K.~K.},
       title={Free vibrations of single-branch structural systems},
        date={1969},
     journal={J I Math Appl},
      volume={5},
      number={3},
       pages={351\ndash 362},
}

\bib{gupta70}{article}{
      author={Gupta, K.~K.},
       title={Vibration of frames and other structures with banded stiffness
  matrix},
        date={1970},
     journal={Int J Numer Meth Eng},
      volume={2},
      number={2},
       pages={221\ndash 228},
}

\bib{martin67}{article}{
      author={Martin, R.~S.},
      author={Wilkinson, J.~H.},
       title={Solution of symmetric and unsymmetric band equations and the
  calculation of eigenvectors of band matrices},
        date={1967},
     journal={Numer Math},
      volume={9},
      number={4},
       pages={279\ndash 301},
}

\bib{scott84}{article}{
      author={Scott, David~S.},
       title={Computing a few eigenvalues and eigenvectors of a symmetric band
  matrix},
        date={1984},
     journal={SIAM J Sci Comput},
      volume={5},
      number={3},
       pages={658\ndash 666},
}

\bib{jess82}{article}{
      author={Jess, Jochen A.~G.},
      author={Kees, H. G.~M.},
       title={A data structure for parallel {L/U} decomposition},
        date={1982March},
     journal={IEEE T Comput},
      volume={C-31},
      number={3},
       pages={231\ndash 239},
}

\bib{liu93}{article}{
      author={Liu, Joseph W.~H.},
      author={Ng, Esmond~G.},
      author={Peyton, Barry~W.},
       title={On finding supernodes for sparse matrix computations},
        date={1993},
     journal={SIAM J Matrix Anal A},
      volume={14},
      number={1},
       pages={242\ndash 252},
}

\bib{chen08}{article}{
      author={Chen, Yanqing},
      author={Davis, Timothy~A.},
      author={Hager, William~W.},
      author={Rajamanickam, Sivasankaran},
       title={Algorithm 887: {CHOLMOD}, supernodal sparse {Cholesky}
  factorization and update/downdate},
        date={2008},
     journal={ACM T Math Software},
      volume={35},
      number={3},
}

\bib{duff04}{article}{
      author={Duff, Iain~S.},
       title={{\tt MA57}---a code for the solution of sparse symmetric definite
  and indefinite systems},
        date={2004},
     journal={ACM T Math Software},
      volume={30},
       pages={118\ndash 144},
}

\bib{gupta00}{techreport}{
      author={Gupta, Anshul},
      author={Avron, Haim},
       title={{WSMP}: Watson sparse matrix package},
 institution={IBM T.J. Watson Research Center},
        date={2000},
      number={RC 21886 (98462)},
}

\bib{hogg10}{article}{
      author={Hogg, J.~D.},
      author={Reid, J.~K.},
      author={Scott, J.~A.},
       title={Design of a multicore sparse {Cholesky} factorization using
  {DAGs}},
        date={2010},
     journal={SIAM J Sci Comput},
      volume={32},
      number={6},
       pages={3627\ndash 3649},
}

\bib{hogg11}{techreport}{
      author={Hogg, Jonathan~D.},
      author={Scott, Jennifer~A.},
       title={{\texttt{HSL\_MA97}}: A bit-compatible multifrontal code for
  sparse symmetric systems},
 institution={Rutherford Appleton Laboratory},
        date={2011},
      number={RAL-TR-2011-024},
}

\bib{li05}{article}{
      author={Li, Xiaoye~S.},
       title={An overview of {SuperLU}: Algorithms, implementation, and user
  interface},
        date={2005-09},
     journal={ACM T Math Software},
      volume={31},
      number={3},
       pages={302\ndash 325},
}

\bib{gilbert94}{article}{
      author={Gilbert, John~R.},
      author={Ng, Esmond~G.},
      author={Peyton, Barry~W.},
       title={An efficient algorithm to compute row and column counts for
  sparse {Cholesky} factorization},
        date={1994},
     journal={SIAM J Matrix Anal A},
      volume={15},
      number={4},
       pages={1075\ndash 1091},
}

\bib{gilbert92}{article}{
      author={Gilbert, John~R.},
      author={Moler, Cleve},
      author={Schreiber, Robert},
       title={Sparse matrices in {MATLAB}: Design and implementation},
        date={1992},
     journal={SIAM J Matrix Anal A},
      volume={13},
      number={1},
       pages={333\ndash 356},
}

\bib{gilbert99}{misc}{
      author={Gilbert, John~R.},
      author={Pugh, William~W., Jr.},
      author={Shpeisman, Tatiana},
       title={Ordered sparse accumulator and its use in efficient sparse matrix
  computation},
        date={Nov 9, 1999},
        note={United States patent 5,983,230},
}

\bib{cormen09}{book}{
      author={Cormen, Thomas~H.},
      author={Leiserson, Charles~E.},
      author={Rivest, Ronald~L.},
      author={Stein, Clifford},
       title={Introduction to algorithms},
     edition={3},
   publisher={MIT Press},
        date={2009},
}

\bib{higham02}{book}{
      author={Higham, Nicholas~J.},
       title={Accuracy and stability of numerical algorithms},
     edition={2},
   publisher={SIAM},
        date={2002},
}

\bib{sorensen85}{article}{
      author={Sorensen, Danny~C.},
       title={Analysis of pairwise pivoting in {Gaussian} elimination},
        date={1985},
     journal={IEEE T Comput},
      volume={C-34},
       pages={274\ndash 278},
}

\bib{grigori11}{article}{
      author={Grigori, Laura},
      author={Demmel, James~W.},
      author={Xiang, Hua},
       title={{CALU}: A communication optimal {LU} factorization algorithm},
        date={2011},
     journal={SIAM J Matrix Anal A},
      volume={32},
      number={4},
       pages={1317\ndash 1350},
}

\bib{davis11b}{article}{
      author={Davis, Timothy~A.},
      author={Hu, Yifan},
       title={The {University of Florida} sparse matrix collection},
        date={2011},
     journal={ACM T Math Software},
      volume={38},
       pages={1:1\ndash 1:25},
}

\bib{duff05}{article}{
      author={Duff, Iain~S.},
      author={Pralet, St\'ephane},
       title={Strategies for scaling and pivoting for sparse symmetric
  indefinite problems},
        date={2005},
     journal={SIAM J Matrix Anal A},
      volume={27},
      number={2},
       pages={313\ndash 340},
}

\end{biblist}
\end{bibdiv}
\end{document}